\documentclass{amsart}
\usepackage{amsmath}
\usepackage{amsfonts}

\newtheorem{theorem}{Theorem}
\theoremstyle{plain}

\newtheorem{corollary}{Corollary}

\newtheorem{example}{Example}

\newtheorem{lemma}{Lemma}

\newtheorem{remark}{Remark}

\numberwithin{equation}{section}

\begin{document}
\title[Logarithmic derivative near a singular point and applications]{%
Estimates of the logarithmic derivative near a singular point and applications}
\author{Saada Hamouda}
\address{Laboratory of pure and applied mathematics\\
Department of mathematics, UMAB university\\
Algeria}
\email{saada.hamouda@univ-mosta.dz}

\begin{abstract}
In this paper, we will give estimates for the logarithmic derivative $%
\left\vert \frac{f^{\left( k\right) }\left( z\right) }{f\left( z\right) }%
\right\vert $ where $f$ is a meromorphic function in a region of the form $%
D\left( 0,R\right) =\left\{ z\in 
\mathbb{C}
:0<\left\vert z\right\vert <R\right\} .$ Some applications on the growth of
solutions of linear differential equations near a singular point are given.
\end{abstract}

\maketitle

\section{Introduction and statement of results}

Throughout this paper, we assume that the reader is familiar with the
fundamental results and the standard notations of the Nevanlinna value
distribution theory of meromorphic function on the complex plane $%
\mathbb{C}
$ and in the unit disc $D=\left\{ z\in 
\mathbb{C}
:\left\vert z\right\vert <1\right\} \ $(see \cite{haym, yang, lain}). The
importance of this theory has inspired many authors to find modifications
and generalizations to different domains. Extensions of Nevanlinna Theory to
annuli have been made by \cite{bieb, khri, kond, korh, mark}. Recently in 
\cite{fet,hamj}, Hamouda and Fettouch investigated the growth of solutions
of a class of linear differential equations 
\begin{equation}
f^{\left( k\right) }+A_{k-1}\left( z\right) f^{\left( k-1\right)
}+...+A_{1}\left( z\right) f^{\prime }+A_{0}\left( z\right) f=0  \label{eq00}
\end{equation}%
near a singular point where the coefficients $A_{j}\left( z\right) \ \left(
j=0,1,...,k-1\right) $ are meromorphic or analytic in $\overline{%
\mathbb{C}
}-\left\{ z_{0}\right\} $ and for that they gave estimates of the
logarithmic derivative $\left\vert \frac{f^{\left( k\right) }\left( z\right) 
}{f\left( z\right) }\right\vert $ for a meromorphic function $f$ in $%
\overline{%
\mathbb{C}
}-\left\{ z_{0}\right\} ,\ \left( \overline{%
\mathbb{C}
}=%
\mathbb{C}
\cup \left\{ \infty \right\} \right) $. A question was asked in \cite{fet}
as the following : can we get similar estimates of $\left\vert \frac{%
f^{\left( k\right) }\left( z\right) }{f\left( z\right) }\right\vert $ in a
region of the form $D_{z_{0}}\left( 0,R\right) =\left\{ z\in 
\mathbb{C}
:0<\left\vert z-z_{0}\right\vert <R\right\} ?$ Naturally, this allows us to
study the solutions of (\ref{eq00}) with meromorphic coefficients in $%
D_{z_{0}}\left( 0,R\right) .$The same question was asked in \cite{hamj} for
another problem concerning the case when the coefficients of (\ref{eq00})
are analytic in $\overline{%
\mathbb{C}
}-\left\{ z_{0}\right\} ,$ the solutions may be non analytic in $\overline{%
\mathbb{C}
}-\left\{ z_{0}\right\} $. In this paper, we will answer this question and
give some applications. Without lose of generality, we will study the case $%
z_{0}=0$ and for $z_{0}\neq 0$ we may use the change of variable $w=z-z_{0}$.

Throughout this paper, we will use the following notation:%
\begin{equation*}
D\left( R_{1},R_{2}\right) =\left\{ z\in 
\mathbb{C}
:R_{1}<\left\vert z\right\vert <R_{2}\right\} ,\ D\left( R\right) =\left\{
z\in 
\mathbb{C}
:\left\vert z\right\vert <R\right\} .
\end{equation*}%
We recall the appropriate definitions \cite{fet, korh, mark}. Suppose that $%
f\left( z\right) $ is meromorphic in $\overline{%
\mathbb{C}
}-\left\{ 0\right\} $. Define the counting function near $0$ by%
\begin{equation}
N_{0}\left( r,f\right) =-\int\limits_{\infty }^{r}\frac{n\left( t,f\right)
-n\left( \infty ,f\right) }{t}dt-n\left( \infty ,f\right) \log r,  \label{d1}
\end{equation}%
where $n\left( t,f\right) $ counts the number of poles of $f\left( z\right) $
in the region $\left\{ z\in 
\mathbb{C}
:t\leq \left\vert z\right\vert \right\} \cup \left\{ \infty \right\} $ each
pole according to its multiplicity; and the proximity function by 
\begin{equation}
m_{0}\left( r,f\right) =\frac{1}{2\pi }\int\limits_{0}^{2\pi }\ln
^{+}\left\vert f\left( re^{i\varphi }\right) \right\vert d\varphi .
\label{d2}
\end{equation}%
The characteristic function of $f$ is defined by%
\begin{equation}
T_{0}\left( r,f\right) =m_{0}\left( r,f\right) +N_{0}\left( r,f\right) .
\label{d3}
\end{equation}%
For a meromorphic function $f\left( z\right) $ in $D\left( 0,R\right) ,$ we
define the counting function near $0$ by%
\begin{equation}
N_{0}\left( r,R^{\prime },f\right) =\int\limits_{r}^{R^{\prime }}\frac{%
n\left( t,f\right) }{t}dt,  \label{0d1}
\end{equation}%
where $n\left( t,f\right) $ counts the number of poles of $f\left( z\right) $
in the region $\left\{ z\in 
\mathbb{C}
:t\leq \left\vert z\right\vert \leq R^{\prime }\right\} $ $\left(
0<R^{\prime }<R\right) ,$ each pole according to its multiplicity; and the
proximity function near the singular point $0$ by 
\begin{equation}
m_{0}\left( r,f\right) =\frac{1}{2\pi }\int\limits_{0}^{2\pi }\ln
^{+}\left\vert f\left( re^{i\varphi }\right) \right\vert d\varphi .
\label{0d2}
\end{equation}%
The characteristic function of $f$ is defined in the usual manner by%
\begin{equation}
T_{0}\left( r,R^{\prime },f\right) =m_{0}\left( r,f\right) +N_{0}\left(
r,R^{\prime },f\right) .  \label{0d3}
\end{equation}

In addition, the order of growth of a meromorphic function $f\left( z\right) 
$ near $0$ is defined by%
\begin{equation}
\sigma _{T}\left( f,0\right) =\underset{r\rightarrow 0}{\lim \sup }\frac{%
\log ^{+}T_{0}\left( r,R^{\prime },f\right) }{-\log r}.  \label{d4}
\end{equation}%
For an analytic function $f\left( z\right) $ in $D\left( 0,R\right) ,$ we
have also the definition%
\begin{equation}
\sigma _{M}\left( f,0\right) =\underset{r\rightarrow 0}{\lim \sup }\frac{%
\log ^{+}\log ^{+}M_{0}\left( r,f\right) }{-\log r},  \label{d5}
\end{equation}%
where $M_{0}\left( r,f\right) =\max \left\{ \left\vert f\left( z\right)
\right\vert :\left\vert z\right\vert =r\right\} .$

If $f\left( z\right) $ is meromorphic in $D\left( 0,R\right) $ of finite
order $0<\sigma _{T}\left( f,0\right) =\sigma <\infty $, then we can define
the type of $f$ as the following:%
\begin{equation*}
\tau _{T}\left( f,0\right) =\underset{r\rightarrow 0}{\lim \sup \ }r^{\sigma
}T_{0}\left( r,R^{\prime },f\right) .
\end{equation*}%
If $f\left( z\right) $ is analytic in $D\left( 0,R\right) $ of finite order $%
0<\sigma _{M}\left( f,0\right) =\sigma <\infty $, we have also another
definition of the type of $f$ as the following:%
\begin{equation}
\tau _{M}\left( f,0\right) =\underset{r\rightarrow 0}{\lim \sup \ }r^{\sigma
}\log ^{+}M_{0}\left( r,f\right) .  \label{d5b}
\end{equation}%
By the usual manner, we define the iterated order near $0$ as follows:

\begin{equation}
\sigma _{n,T}\left( f,0\right) =\underset{r\rightarrow 0}{\lim \sup }\frac{%
\log _{n}^{+}T_{0}\left( r,R^{\prime },f\right) }{-\log r},  \label{d6}
\end{equation}%
\begin{equation}
\sigma _{n,M}\left( f,0\right) =\underset{r\rightarrow 0}{\lim \sup }\frac{%
\log _{n+1}^{+}M_{0}\left( r,f\right) }{-\log r},  \label{d7}
\end{equation}%
where $\log _{1}^{+}x=\log ^{+}x=\max \left\{ \log x,0\right\} $ and $\log
_{n}^{+}x=\log ^{+}\log _{n-1}^{+}x$ for $n\geq 2.$

\begin{remark}
The choice of $R^{\prime }$ in (\ref{d1}) does not have any influence in the
values $\sigma _{T}\left( f,0\right) $ and $\tau _{T}\left( f,0\right) .$ In
fact, if we take two values of $R^{\prime },$ namely $0<R_{1}^{\prime
}<R_{2}^{\prime }<R,$ then we have%
\begin{equation*}
\int\limits_{R_{1}^{\prime }}^{R_{2}^{\prime }}\frac{n\left( t,f\right) }{t}%
dt=n\log \frac{R_{2}^{\prime }}{R_{1}^{\prime }},
\end{equation*}%
where $n$ designates the number of poles of $f\left( z\right) $ in the
region $\left\{ z\in 
\mathbb{C}
:R_{1}^{\prime }\leq \left\vert z\right\vert \leq R_{2}^{\prime }\right\} \ $%
which is bounded. Thus, $T_{0}\left( r,R_{1}^{\prime },f\right) =T_{0}\left(
r,R_{2}^{\prime },f\right) +C$ where $C$ is real constant. So, we can write
briefly $T_{0}\left( r,f\right) $ instead of $T_{0}\left( r,R^{\prime
},f\right) .$
\end{remark}

\begin{example}
Consider the function $f\left( z\right) =\exp \left\{ z^{2}+\frac{1}{z^{2}}%
\right\} .$ We have 
\begin{equation*}
T_{0}\left( r,f\right) =m_{0}\left( r,f\right) =\frac{1}{\pi }\left( r^{2}+%
\frac{1}{r^{2}}\right) ,
\end{equation*}
then $\sigma _{T}\left( f,0\right) =2,\ \tau _{T}\left( f,0\right) =\frac{1}{%
\pi }.$ Also we have%
\begin{equation*}
M_{0}\left( r,f\right) =\exp \left\{ r^{2}+\frac{1}{r^{2}}\right\} ,
\end{equation*}
then $\sigma _{M}\left( f,0\right) =2,\ \tau _{M}\left( f,0\right) =1.$
\end{example}

The main tool we use throughout this paper is the decomposition lemma of G.
Valiron.

\begin{lemma}
\label{val}\cite{val, mark} (Valiron's decomposition lemma) Let $f$ be
meromorphic function in $D\left( R_{1},R_{2}\right) ,$ and set $%
R_{1}<R^{\prime }<R_{2}.$ Then $f$ may be represented as%
\begin{equation*}
f\left( z\right) =z^{m}\phi \left( z\right) \mu \left( z\right)
\end{equation*}%
where\newline
a) The poles and zeros of $f$ in $D\left( R_{1},R^{\prime }\right) $ are
precisely the poles and zeros of $\phi \left( z\right) .$ The poles and
zeros of $f$ in $D\left( R^{\prime },R_{2}\right) $ are precisely the poles
and zeros of $\mu \left( z\right) .$\newline
b) $\phi \left( z\right) $ is meromorphic in $D\left( R_{1},\infty \right) $%
\ and analytic and nonzero in $D\left[ R^{\prime },\infty \right] .$\newline
c) $\phi \left( z\right) $ satisfies%
\begin{equation*}
\left\vert \frac{\phi ^{\prime }\left( \xi e^{i\theta }\right) }{\phi \left(
\xi e^{i\theta }\right) }\right\vert =O\left( \frac{1}{\xi ^{2}}\right) ,\
\xi \rightarrow \infty .
\end{equation*}%
d) $\mu \left( z\right) $ is meromorphic in $D\left( R\right) $ and analytic
and nonzero in $D\left( R^{\prime }\right) .$\newline
e) $m\in 
\mathbb{Z}
.$
\end{lemma}

\begin{remark}
\label{rem1}Let $f$ be a non constant meromorphic function in $D\left(
0,R\right) $ and $f\left( z\right) =z^{m}\phi \left( z\right) \mu \left(
z\right) $ is a Valiron's decomposition. Set $\tilde{\phi}\left( z\right)
=z^{m}\phi \left( z\right) .$ It is easy to see that%
\begin{equation}
T_{0}\left( r,f\right) =T_{0}\left( r,\tilde{\phi}\right) +O\left( 1\right) .
\label{re1}
\end{equation}%
If $f$ be a non constant analytic function in $D\left( 0,R\right) ,$ then $%
\tilde{\phi}\left( z\right) $ is analytic in $D\left( 0,\infty \right] $ and
by \cite{fet} and (\ref{re1}), we obtain that $\sigma _{n,T}\left(
f,0\right) =\sigma _{n,M}\left( f,0\right) $ for $n\geq 1.$
\end{remark}

Now, we give estimates on the logarithmic derivative of a meromorphic function
in $D\left( 0,R\right) .$

\begin{theorem}
\label{th1}Let $f$ be meromorphic function in $D\left( 0,R\right) $ with a
singular point at the origin and let $\alpha >0$, then \newline
(i) there exists a set $E_{1}^{\ast }\subset \left( 0,R^{\prime }\right) \
\left( 0<R^{\prime }<R\right) $ that has finite logarithmic measure $%
\int\limits_{0}^{R^{\prime }}\frac{\chi _{E_{1}^{\ast }}\left( t\right) }{t}%
dt<\infty $ and a constant $C>0$ such that for all $r=\left\vert
z\right\vert $ satisfying $r\in \left( 0,R^{\prime }\right) \backslash
E_{1}^{\ast }$, we have%
\begin{equation}
\left\vert \frac{f^{\left( k\right) }\left( z\right) }{f\left( z\right) }%
\right\vert \leq C\left[ \frac{1}{r}T_{0}\left( \frac{r}{\alpha },f\right)
\log ^{\alpha }\left( \frac{1}{r}\right) \log T_{0}\left( \frac{r}{\alpha }%
,f\right) \right] ^{k}\ \left( k\in 
\mathbb{N}
\right) ;  \label{l3}
\end{equation}%
(ii) there exists a set $E_{2}^{\ast }\subset \left[ 0,2\pi \right) $ that
has a linear measure zero such that for all $\theta \in \left[ 0,2\pi
\right) \backslash E_{2}^{\ast }$ there exists a constant $r_{0}=r_{0}\left(
\theta \right) >0$ such that(\ref{l3}) holds for all $z$ satisfying $\arg
z\in \left[ 0,2\pi \right) \backslash E_{2}^{\ast }$ and $r=\left\vert
z\right\vert <r_{0}$.
\end{theorem}

The following two corollaries are consequences of Theorem \ref{th1} and have
independent interest.

\begin{corollary}
\label{coro1}Let $f$ be a non constant meromorphic function in $D\left(
0,R\right) $ with a singular point at the origin of finite order $\sigma
\left( f,0\right) =\sigma <\infty $; let $\varepsilon >0$ be a given
constant. Then the following two statements hold.\newline
i) There exists a set $E_{1}^{\ast }\subset \left( 0,R^{\prime }\right) $
that has finite logarithmic measure $\int\limits_{0}^{R^{\prime }}\frac{\chi
_{E_{1}^{\ast }}\left( t\right) }{t}dt<\infty $ such that for all $%
r=\left\vert z\right\vert $ satisfying $r\in \left( 0,R^{\prime }\right)
\backslash E_{1}^{\ast }$, we have%
\begin{equation}
\left\vert \frac{f^{\left( k\right) }\left( z\right) }{f\left( z\right) }%
\right\vert \leq \frac{1}{r^{k\left( \sigma +1+\varepsilon \right) }},\
\left( k\in 
\mathbb{N}
\right) .  \label{l7}
\end{equation}%
ii) There exists a set $E_{2}^{\ast }\subset \left[ 0,2\pi \right) $ that
has a linear measure zero such that for all $\theta \in \left[ 0,2\pi
\right) \backslash E_{2}^{\ast }$ there exists a constant $r_{0}=r_{0}\left(
\theta \right) >0$ such that for all $z$ satisfying $\arg \left( z\right)
\in \left[ 0,2\pi \right) \backslash E_{2}^{\ast }$ and $r=\left\vert
z\right\vert <r_{0}$ the inequality (\ref{l7}) holds.
\end{corollary}

\begin{corollary}
\label{coro2}Let $f$ be a non constant meromorphic function in $D\left(
0,R\right) $ with a singular point at the origin of finite iterated order $%
\sigma _{n}\left( f,0\right) =\sigma <\infty \ \left( n\geq 2\right) $; let $%
\varepsilon >0$ be a given constant. Then the following two statements hold.%
\newline
i) There exists a set $E_{1}^{\ast }\subset \left( 0,R^{\prime }\right) $
that has finite logarithmic measure $\int\limits_{0}^{R^{\prime }}\frac{\chi
_{E_{1}^{\ast }}\left( t\right) }{t}dt<\infty $ such that for all $%
r=\left\vert z\right\vert $ satisfying $r\in \left( 0,R^{\prime }\right)
\backslash E_{1}^{\ast }$, we have%
\begin{equation*}
\left\vert \frac{f^{\left( k\right) }\left( z\right) }{f\left( z\right) }%
\right\vert \leq \exp _{n-1}\left\{ \frac{1}{r^{\sigma +\varepsilon }}%
\right\} ,\ \left( k\in 
\mathbb{N}
\right) .
\end{equation*}%
ii) There exists a set $E_{2}^{\ast }\subset \left[ 0,2\pi \right) $ that
has a linear measure zero such that for all $\theta \in \left[ 0,2\pi
\right) \backslash E_{2}^{\ast }$ there exists a constant $r_{0}=r_{0}\left(
\theta \right) >0$ such that for all $z$ satisfying $\arg \left( z\right)
\in \left[ 0,2\pi \right) \backslash E_{2}^{\ast }$ and $r=\left\vert
z\right\vert <r_{0}$ the inequality (\ref{l7}) holds.
\end{corollary}

As applications of Theorem \ref{th1}, we have the following results.

\begin{theorem}
\label{th2b}Let $A_{0}\left( z\right) \not\equiv 0,A_{1}\left( z\right)
,...,A_{k-1}\left( z\right) $ be analytic functions in $D\left( 0,R\right) $%
. All solutions $f$ of 
\begin{equation}
f^{\left( k\right) }+A_{k-1}\left( z\right) f^{\left( k-1\right)
}+...+A_{1}\left( z\right) f^{\prime }+A_{0}\left( z\right) f=0  \label{eq3}
\end{equation}%
satisfy $\sigma _{n+1}\left( f,0\right) \leq \alpha $ if and only if $\sigma
_{n}\left( A_{j},0\right) \leq \alpha $ for all $\left( j=0,1,...,k-1\right) 
$, where $n$ is a positive integer. Moreover, if $q\in \left\{
0,1,...,k-1\right\} $ is the largest index for which $\sigma _{n}\left(
A_{q},0\right) =\underset{0\leq j\leq k-1}{\max }\left\{ \sigma _{n}\left(
A_{j},0\right) \right\} $ then there are at least $k-q$ linearly independent
solutions $f$ of (\ref{eq3}) such that $\sigma _{n+1}\left( f,0\right)
=\sigma _{n}\left( A_{q},0\right) .$
\end{theorem}

Similar result to Theorem \ref{th2b} in the unit disc has been given by \cite%
[Theorem 1.1]{heit}.

\begin{corollary}
\label{cor1}Let $A_{0}\left( z\right) \not\equiv 0,A_{1}\left( z\right)
,...,A_{k-1}\left( z\right) $ be analytic functions in $D\left( 0,R\right) $
satisfying $\sigma _{n}\left( A_{j},0\right) <\sigma _{n}\left(
A_{0},0\right) <\infty \ \left( j=1,...,k-1\right) .$ Then, every solution $%
f\left( z\right) \not\equiv 0$ of (\ref{eq3}) satisfies $\sigma _{n+1}\left(
f,0\right) =\sigma _{n}\left( A_{0},0\right) .$
\end{corollary}

\begin{corollary}
Let $b\neq 0$ be complex constants and $n$ be a positive integer. Let $%
A\left( z\right) ,B\left( z\right) \not\equiv 0$ be analytic functions in $%
D\left( 0,R\right) $ with $\max \left\{ \sigma \left( A,0\right) ,\sigma
\left( B,0\right) \right\} <n.$ Then, every solution $f\left( z\right)
\not\equiv 0$ of the differential equation%
\begin{equation}
f^{\prime \prime }+A\left( z\right) f^{\prime }+B\left( z\right) \exp
\left\{ \frac{b}{z^{n}}\right\} f=0,  \label{eqc}
\end{equation}%
satisfies $\sigma _{2}\left( f,0\right) =n.$
\end{corollary}

\begin{example}
Every solution $f\left( z\right) \not\equiv 0$ of the differential equation%
\begin{equation}
f^{\prime \prime }+\exp \left\{ \frac{1}{\left( 1-z\right) ^{m}}\right\}
f^{\prime }+\exp \left\{ \frac{1}{z^{n}}\right\} f=0,  \label{eqe}
\end{equation}%
satisfies $\sigma _{2}\left( f,0\right) =n,$ where $m$ and $n$ are positive
integers.
\end{example}

Similar equations to (\ref{eqc}) and (\ref{eqe}) with analytic coefficients in the unit disc are investigated in \cite%
{ham12}.

Now, we will study the case when $\sigma \left( A_{j},0\right) =\sigma
\left( A_{0},0\right) $ for some $j\neq 0.$

\begin{theorem}
\label{th2}Let $A_{0}\left( z\right) \not\equiv 0,A_{1}\left( z\right)
,...,A_{k-1}\left( z\right) $ be analytic functions in $D\left( 0,R\right) $
satisfying $0<\sigma \left( A_{j},0\right) \leq \sigma \left( A_{0},0\right)
<\infty $ and 
\begin{equation*}
\max \left\{ \tau _{M}\left( A_{j},0\right) :\sigma \left( A_{j},0\right)
=\sigma \left( A_{0},0\right) \right\} <\tau _{M}\left( A_{0},0\right) \
\left( j=1,...,k-1\right) .
\end{equation*}%
Then, every solution $f\left( z\right) \not\equiv 0$ of (\ref{eq3})
satisfies $\sigma _{2}\left( f,0\right) =\sigma \left( A_{0},0\right) .$
\end{theorem}

The analogous of this result in the complex plane and in the unit disc are investigated in \cite%
{tu, ham}.

\begin{theorem}
\label{th3}Let $a,b\neq 0$ be complex constants such that $\arg a\neq \arg b$
or $a=cb$ $\left( 0<c<1\right) $ and $n$ be a positive integer. Let $A\left(
z\right) ,B\left( z\right) \not\equiv 0$ be analytic functions in $D\left(
0,R\right) $ with $\max \left\{ \sigma \left( A,0\right) ,\sigma \left(
B,0\right) \right\} <n.$ Then, every solution $f\left( z\right) \not\equiv 0$
of the differential equation%
\begin{equation}
f^{\prime \prime }+A\left( z\right) \exp \left\{ \frac{a}{z^{n}}\right\}
f^{\prime }+B\left( z\right) \exp \left\{ \frac{b}{z^{n}}\right\} f=0,
\label{eq4}
\end{equation}%
satisfies $\sigma _{2}\left( f,0\right) =n.$
\end{theorem}

Similar results to Theorem \ref{th3} are established in different situations
in \cite{chen1, ham12, fet}.

\begin{example}
By Theorem \ref{th3}, every solution $f\left( z\right) \not\equiv 0$ of the
differential equation%
\begin{equation*}
f^{\prime \prime }+\exp \left\{ \frac{i}{z\left( z+1\right) }\right\}
f^{\prime }+\exp \left\{ \frac{1}{z\left( z-1\right) ^{2}}\right\} f=0,
\end{equation*}%
satisfies $\sigma _{2}\left( f,0\right) =1$ and $\sigma _{2}\left(
f,1\right) =2.$
\end{example}

\section{Preliminaries lemmas}

To prove these results we need the following lemmas.

\begin{lemma}
\label{lem1}\cite{gund} Let $g$ be a transcendental meromorphic function in $%
\mathbb{C}
$, and let $\alpha >0$ $\varepsilon >0$ be given real constants; then\newline
i) there exists a set $E_{1}\subset \left( 1,\infty \right) $ that has a
finite logarithmic measure and a constant $c>0$ that depends only on $\alpha 
$ such that for all $R=\left\vert w\right\vert $ satisfying $R\notin \left[
0,1\right) \cup E_{1}$, we have%
\begin{equation}
\left\vert \frac{g^{\left( k\right) }\left( w\right) }{g\left( w\right) }%
\right\vert \leq c\left[ T\left( \alpha R,g\right) \frac{\log ^{\alpha
}\left( R\right) }{R}\log T\left( \alpha R,g\right) \right] ^{k};  \label{l1}
\end{equation}%
ii) there exists a set $E_{2}\subset \left[ 0,2\pi \right) $ that has a
linear measure zero such that for all $\theta \in \left[ 0,2\pi \right)
\backslash E_{2}$ there exists a constant $R_{0}=R_{0}\left( \theta \right)
>0$ such that (\ref{l1}) holds for all $z$ satisfying $\arg z\in \left[
0,2\pi \right) \backslash E_{2}$ and $r=\left\vert z\right\vert >R_{0}.$
\end{lemma}

\begin{lemma}
\label{lem2}\cite{fet} Let $\phi $ be a non constant meromorphic function in 
$D\left( 0,\infty \right] $ and set $g\left( w\right) =\phi \left( \frac{1}{w%
}\right) $. Then, $g\left( w\right) $ is meromorphic in $%
\mathbb{C}
$ and we have%
\begin{equation*}
T\left( \frac{1}{r},g\right) =T_{0}\left( r,\phi \right) ,
\end{equation*}%
and so $\sigma \left( f,0\right) =\sigma \left( g\right) .$
\end{lemma}

\begin{lemma}
\label{lem5}Let $f$ be a non constant analytic function in $D\left(
0,R\right) $ of finite order $\sigma \left( f,0\right) =\sigma >0$ and a
finite type $\tau \left( f,0\right) =\tau >0$. Then, for any given $0<\beta
<\tau $ there exists a set $F\subset \left( 0,1\right) $ of infinite
logarithmic measure such that for all $r\in F$ we have%
\begin{equation*}
\log M_{0}\left( r,f\right) >\frac{\beta }{r^{\sigma }},
\end{equation*}%
where $M_{0}\left( r,f\right) =\max \left\{ \left\vert f\left( z\right)
\right\vert :\left\vert z\right\vert =r\right\} .$
\end{lemma}

\begin{proof}
By the definition of $\tau \left( f,0\right) ,$ there exists a decreasing
sequence $\left\{ r_{m}\right\} \rightarrow 0$ satisfying $\frac{m}{m+1}%
r_{m}>r_{m+1}$ and%
\begin{equation*}
\underset{m\rightarrow \infty }{\lim }r_{m}^{\sigma }\log M_{0}\left(
r_{m},f\right) =\tau .
\end{equation*}%
Then, there exists $m_{0}$ such that for all $m>m_{0}$ and for a given $%
\varepsilon >0$, we have%
\begin{equation}
\log M_{0}\left( r_{m},f\right) >\frac{\tau -\varepsilon }{r_{m}^{\sigma }}.
\label{l8}
\end{equation}%
There exists $m_{1}$ such that for all $m>m_{1}$ and for a given $%
0<\varepsilon <\tau -\beta $, we have%
\begin{equation}
\left( \frac{m}{m+1}\right) ^{\sigma }>\frac{\beta }{\tau -\varepsilon }
\label{9}
\end{equation}%
By (\ref{l8}) and (\ref{9}), for all $m>m_{2}=\max \left\{
m_{0},m_{1}\right\} $ and for any $r\in \left[ \frac{m}{m+1}r_{m},r_{m}%
\right] $, we have%
\begin{equation*}
\log M_{0}\left( r,f\right) >\log M_{0}\left( r_{m},f\right) >\frac{\tau
-\varepsilon }{r_{m}^{\sigma }}>\frac{\tau -\varepsilon }{r^{\sigma }}\left( 
\frac{m}{m+1}\right) ^{\sigma }>\frac{\beta }{r^{\sigma }}.
\end{equation*}%
Set $F=\bigcup\limits_{m=m_{2}}^{\infty }\left[ \frac{m}{m+1}r_{m},r_{m}%
\right] ;$ then we have%
\begin{equation*}
\sum\limits_{m=m_{2}}^{\infty }\int\limits_{\frac{m}{m+1}r_{m}}^{r_{m}}\frac{%
dt}{t}=\sum\limits_{m>m_{2}}\log \frac{m+1}{m}=\infty .
\end{equation*}
\end{proof}

By the same method of the proof of Lemma \ref{lem5}, we can prove the
following lemma.

\begin{lemma}
\label{lem6}Let $f$ be a non constant analytic function in $D\left(
0,R\right) $ of order $\sigma \left( f,0\right) >\alpha >0$. Then there
exists a set $F\subset \left( 0,1\right) $ of infinite logarithmic measure
such that for all $r\in F$ we have%
\begin{equation*}
\log M_{0}\left( r,f\right) >\frac{1}{r^{\alpha }}.
\end{equation*}
\end{lemma}

\begin{lemma}
\label{lem6b}\cite[Theorem 8]{hamj} Let $f$ be non constant analytic
function in $\overline{%
\mathbb{C}
}-\left\{ z_{0}\right\} .$ Then, there exists a set $E\subset \left(
0,1\right) $ that has finite logarithmic measure, that is $%
\int\limits_{0}^{1}\frac{\chi _{E}}{t}dt<\infty ,$ such that for all $%
j=0,1,...,k,$ we have%
\begin{equation*}
\frac{f^{\left( j\right) }\left( z_{r}\right) }{f\left( z_{r}\right) }%
=\left( 1+o\left( 1\right) \right) \left( \frac{V_{z_{0}}\left( r\right) }{%
z_{r}-z_{0}}\right) ^{j},
\end{equation*}%
as $r\rightarrow 0,\ r\notin E,$ where $V_{z_{0}}\left( r\right) $ is the
central index of $f$ and $z_{r}$ is a point in the circle $\left\vert
z_{0}-z\right\vert =r$ that satisfies $\left\vert f\left( z_{r}\right)
\right\vert =\underset{\left\vert z_{0}-z\right\vert =r}{\max }\left\vert
f\left( z\right) \right\vert .$
\end{lemma}

\begin{lemma}
\label{lem7b}Let $f$ be a non constant analytic function in $\overline{%
\mathbb{C}
}-\left\{ z_{0}\right\} $ of iterated order $\sigma _{n}\left(
f,z_{0}\right) =\sigma \ \left( n\geq 2\right) $, and let $V_{z_{0}}\left(
r\right) $ be the central index of $f.$ Then%
\begin{equation}
\underset{r\rightarrow 0}{\lim \sup }\frac{\log _{n}^{+}V_{z_{0}}\left(
r\right) }{-\log r}=\sigma .  \label{700}
\end{equation}
\end{lemma}

\begin{proof}
Set $g\left( w\right) =f\left( z_{0}-\frac{1}{w}\right) .$ Then $g\left(
w\right) $ is entire function of iterated order $\sigma _{n}\left( g\right)
=\sigma _{n}\left( f,z_{0}\right) =\sigma $, and if $V\left( R\right) $
denotes the central index of $g,$ then $V_{z_{0}}\left( r\right) =V\left(
R\right) $ with $R=\frac{1}{r}.$ From \cite[Lemma 2]{chen3}, we have 
\begin{equation}
\underset{R\rightarrow +\infty }{\lim \sup }\frac{\log _{n}^{+}V\left(
R\right) }{\log R}=\sigma .  \label{710}
\end{equation}%
Substituting $\ R$ by $\frac{1}{r}$ in (\ref{710}), we get (\ref{700}).
\end{proof}

\begin{lemma}
\label{lem7}Let $A_{j}\left( z\right) \ \left( j=0,...,k-1\right) $ be
analytic functions in $D\left( 0,R\right) $ such that $0$ is a singular
point for at least one of the coefficients $A_{j}\left( z\right) $ and $%
\sigma _{n}\left( A_{j},0\right) \leq \alpha <\infty $. If $f$ is a solution
of the differential equation 
\begin{equation}
f^{\left( k\right) }+A_{k-1}\left( z\right) f^{\left( k-1\right)
}+...+A_{1}\left( z\right) f^{\prime }+A_{0}\left( z\right) f=0,  \label{80}
\end{equation}%
then $\sigma _{n+1}\left( f,0\right) \leq \alpha .$
\end{lemma}

\begin{proof}
Let $f\not\equiv 0$ be a solution of (\ref{80}). It is clear that $f$ is
analytic in $D\left( 0,R\right) .$ Let $f\left( z\right) =z^{m}\phi \left(
z\right) \mu \left( z\right) $ be a Valiron's decomposition and set $\tilde{%
\phi}\left( z\right) =z^{m}\phi \left( z\right) .$ By Valiron's
decomposition lemma and since $f\left( z\right) $ is analytic function in $%
D\left( 0,R\right) $, $\tilde{\phi}\left( z\right) $ is analytic in $D\left(
0,\infty \right] .$ By Lemma \ref{lem6b}, there exists a set $E\subset
\left( 0,1\right) $ that has finite logarithmic measure, such that for all $%
j=0,1,...,k,$ we have%
\begin{equation}
\frac{\tilde{\phi}^{\left( j\right) }\left( z_{r}\right) }{\tilde{\phi}%
\left( z_{r}\right) }=\left( 1+o\left( 1\right) \right) \left( \frac{%
V_{0}\left( r\right) }{z_{r}}\right) ^{j},  \label{70}
\end{equation}%
as $r\rightarrow 0,\ r\notin E,$ where $V_{0}\left( r\right) $ is the
central index of $f$ near the singular point $0,$ $z_{r}$ is a point in the
circle $\left\vert z\right\vert =r$ that satisfies $\left\vert f\left(
z_{r}\right) \right\vert =\underset{\left\vert z\right\vert =r}{\max }%
\left\vert f\left( z\right) \right\vert .$ Since $\mu \left( z\right) $ is
analytic and non zero in $D\left( R^{\prime }\right) ,$ we have%
\begin{equation}
\left\vert \frac{\mu ^{\left( j\right) }\left( z\right) }{\mu \left(
z\right) }\right\vert \leq M,\ \left( j\in 
\mathbb{N}
\right) .  \label{71}
\end{equation}%
Set $M_{0}\left( r\right) =\underset{\left\vert z\right\vert =r}{\max }%
\left\{ \left\vert A_{j}\left( z\right) \right\vert :j=0,1,...,k-1\right\} .$
From (\ref{80}), we can write%
\begin{equation}
\frac{f^{(k)}}{f}+A_{k-1}(z)\frac{f^{(k-1)}}{f}+\dots +A_{1}(z)\frac{%
f^{\prime }}{f}+A_{0}(z)=0.  \label{72}
\end{equation}%
We have $f\left( z\right) =\tilde{\phi}\left( z\right) \mu \left( z\right) $%
, and then%
\begin{equation}
\frac{f^{\left( j\right) }\left( z\right) }{f\left( z\right) }%
=\sum\limits_{i=0}^{i=j}C_{j}^{i}\frac{\tilde{\phi}^{\left( j-i\right)
}\left( z\right) }{\tilde{\phi}\left( z\right) }\frac{\mu ^{\left( i\right)
}\left( z\right) }{\mu \left( z\right) },\ j=1,...,k,  \label{73}
\end{equation}%
where $C_{k}^{j}=\frac{k!}{j!\left( k-j\right) !}$ is the binomial
coefficient. By combining (\ref{70}), (\ref{71}) and (\ref{73}) in (\ref{72}%
), we get%
\begin{equation*}
\left( V_{0}\left( r\right) \right) ^{k}\leq Cr^{k}\left( V_{0}\left(
r\right) \right) ^{k-1}M_{0}\left( r\right) ;
\end{equation*}%
where $r$ near enough to $0$ and $C>0,$ and then%
\begin{equation}
V_{0}\left( r\right) \leq Cr^{k}M_{0}\left( r\right) .  \label{74}
\end{equation}%
By (\ref{74}),\ we obtain $\sigma _{2}\left( f,0\right) \leq \alpha .$
\end{proof}

\begin{lemma}
\label{lem8}Let $A\left( z\right) $ be a non constant analytic function in $%
D\left( 0,R\right) $ with $\sigma \left( A,0\right) <n.$ and let $g\left(
z\right) =A\left( z\right) \exp \left\{ \dfrac{a}{z^{n}}\right\} ,\ $($n\geq
1\ $is an integer) $,a=\alpha +i\beta \neq 0,\ z=re^{i\varphi },\ \delta
_{a}\left( \varphi \right) =\alpha \cos \left( n\varphi \right) +\beta \sin
\left( n\varphi \right) ,$ and $H=\left\{ \varphi \in \left[ 0,2\pi \right)
:\delta _{a}\left( \varphi \right) =0\right\} ,\ $(obviously, $H$ is of
linear measure zero). Then for any\textit{\ given }$\varepsilon >0$ and for
any\textit{\ }$\varphi \in \left[ 0,2\pi \right) \backslash H,$\textit{\
there exists }$r_{0}>0$\textit{\ such that }$0<r<r_{0},$\textit{\ }we have%
\newline
(i) \textit{If }$\delta _{a}\left( \varphi \right) >0,$\textit{\ then}%
\begin{equation}
\exp \left\{ \left( 1-\varepsilon \right) \delta _{a}\left( \varphi \right) 
\frac{1}{r^{n}}\right\} \leq \left\vert g\left( z\right) \right\vert \leq
\exp \left\{ \left( 1+\varepsilon \right) \delta _{a}\left( \varphi \right) 
\frac{1}{r^{n}}\right\} ,  \label{e1}
\end{equation}%
(ii)\textit{\ If }$\delta _{a}\left( \varphi \right) <0,$\textit{\ then}%
\begin{equation}
\exp \left\{ \left( 1+\varepsilon \right) \delta _{a}\left( \varphi \right) 
\frac{1}{r^{n}}\right\} \leq \left\vert g\left( z\right) \right\vert \leq
\exp \left\{ \left( 1-\varepsilon \right) \delta _{a}\left( \varphi \right) 
\frac{1}{r^{n}}\right\} .  \label{e2}
\end{equation}
\end{lemma}

\begin{proof}
Let $A\left( z\right) =z^{m}\phi \left( z\right) \mu \left( z\right) $ be a
Valiron's decomposition and set $\tilde{\phi}\left( z\right) =z^{m}\phi
\left( z\right) .$ By Valiron's decomposition lemma and since $A\left(
z\right) $ is analytic function in $D\left( 0,R\right) $, $\tilde{\phi}%
\left( z\right) $ is analytic in $D\left( 0,\infty \right] .$ By Remark \ref%
{rem1}, $\sigma \left( \tilde{\phi},0\right) =\sigma \left( A,0\right) <n.$
Since $\mu \left( z\right) $ analytic and nonzero in $D\left( R^{\prime
}\right) ,$ we have%
\begin{equation}
0<c_{1}\leq \left\vert \mu \left( z\right) \right\vert \leq c_{2}\text{ as }r%
\text{ is near enough to }0.  \label{e3}
\end{equation}%
By applying \cite[Lemma 2.9]{fet} for $\tilde{\phi}\left( z\right) $, and (%
\ref{e3}), we get (\ref{e1}) and (\ref{e2}).
\end{proof}

Now, we give the standard order reduction procedure of linear differential
equations which is an adaptation of \cite[Lemma 6.4]{4}.

\begin{lemma}
\label{lem11}Let $f_{0,1},f_{0,2},...,f_{0,m}$ be $m\ \left( m\geq 2\right) $
linearly independent meromorphic (in $D\left( 0,R\right) $) solutions of an
equation of the form%
\begin{equation}
y^{\left( k\right) }+A_{0,k-1}\left( z\right) y^{\left( k-1\right)
}+...+A_{0,0}\left( z\right) y=0,\ k\geq m,  \label{5.3}
\end{equation}%
where $A_{0,0}\left( z\right) ,...,A_{0,k-1}\left( z\right) $ are
meromorphic functions in $D\left( 0,R\right) .$ For $1\leq q\leq m-1,$ set%
\begin{equation}
f_{q,j}=\left( \frac{f_{q-1,j+1}}{f_{q-1,1}}\right) ^{\prime },\
j=1,2,...,m-q.  \label{5.4}
\end{equation}%
Then, $f_{q,1},f_{q,2},...,f_{q,m-q}$ are $m-q$ linearly independent
meromorphic (in $D\left( 0,R\right) $) solutions of the equation%
\begin{equation}
y^{\left( k-q\right) }+A_{q,k-q-1}\left( z\right) y^{\left( k-q-1\right)
}+...+A_{q,0}\left( z\right) y=0,  \label{5.5}
\end{equation}%
where%
\begin{equation}
A_{q,j}\left( z\right) =\sum\limits_{i=j+1}^{k-q-1}\left( 
\begin{array}{c}
i \\ 
j+1%
\end{array}%
\right) A_{q-1,j}\left( z\right) \frac{f_{q-1,1}^{\left( i-j-1\right)
}\left( z\right) }{f_{q-1,1}\left( z\right) }  \label{5.6}
\end{equation}%
for $j=0,1,...,k-q-1.$ Here we set $A_{i,k-i}\left( z\right) \equiv 1$ for
all $i=0,1,...,q.$ \newline
Moreover, let $\varepsilon >0$ and suppose for each $j\in \left\{
0,1,...,k-1\right\} $, there exists a real number $\alpha _{j}$ such that%
\begin{equation}
\left\vert A_{0,j}\left( z\right) \right\vert \leq \exp \left\{ \frac{1}{%
r^{\alpha _{j}+\varepsilon }}\right\} ,\ r=\left\vert z\right\vert \notin E.
\label{5.7}
\end{equation}%
Suppose further that each $f_{0,j}$ is of finite hyper-order $\sigma
_{2}\left( f_{0,j},0\right) .$ Set $\beta =\underset{1\leq j\leq m}{\max }%
\left\{ \sigma _{2}\left( f_{0,j},0\right) \right\} $ and $\tau _{p}=%
\underset{p\leq j\leq k-1}{\max }\left\{ \alpha _{j}\right\} $. Then for any
given $\varepsilon >0,$ we have%
\begin{equation}
\left\vert A_{q,j}\left( z\right) \right\vert \leq \exp \left\{ \frac{1}{%
r^{\max \left\{ \tau _{q+j},\beta \right\} +\varepsilon }}\right\} ,\
r=\left\vert z\right\vert \notin E,  \label{5.8}
\end{equation}%
for $j=0,1,...,k-q-1.$
\end{lemma}

\begin{proof}
By \cite[Lemma 6.2 and Lemma 6.3]{4}, we obtain (\ref{5.5}) and (\ref{5.6}).
Therefore, we need only to prove (\ref{5.8}). For this proof, we use
induction on $q.$ First suppose that $q=1$. Then, from (\ref{5.6}) we get%
\begin{equation}
A_{1,j}\left( z\right) =\sum\limits_{i=j+1}^{k}\left( 
\begin{array}{c}
i \\ 
j+1%
\end{array}%
\right) A_{0,i}\left( z\right) \frac{f_{0,1}^{\left( i-j-1\right) }\left(
z\right) }{f_{0,1}\left( z\right) },\ j=0,1,...,k-2.  \label{5.10}
\end{equation}%
Since $\sigma _{2}\left( f_{0,j},0\right) \leq \beta ,$ by Theorem \ref{th1}%
, we have%
\begin{equation}
\left\vert \frac{f_{0,1}^{\left( i-j-1\right) }\left( z\right) }{%
f_{0,1}\left( z\right) }\right\vert \leq \exp \left\{ \frac{1}{r^{\beta
+\varepsilon }}\right\} ,\ r=\left\vert z\right\vert \notin E.  \label{5.10b}
\end{equation}%
It follows from (\ref{5.7}) and (\ref{5.10}) that (\ref{5.8}) holds for $q=1$%
. For the induction step, we make the assumption that (\ref{5.8}) holds for $%
q-1$; i.e.%
\begin{equation}
\left\vert A_{q-1,j}\left( z\right) \right\vert \leq \exp \left\{ \frac{1}{%
r^{\max \left\{ \tau _{q-1+j},\beta \right\} +\varepsilon }}\right\} ,\
r\notin E,  \label{5.11}
\end{equation}%
for $j=1,2,...,k-q-1;$ and we show that (\ref{5.8}) holds for $q$. From (\ref%
{5.6}) we get%
\begin{equation}
A_{q,j}\left( z\right) =\sum\limits_{i=j+1}^{k-q-1}\left( 
\begin{array}{c}
i \\ 
j+1%
\end{array}%
\right) A_{q-1,j}\left( z\right) \frac{f_{q-1,1}^{\left( i-j-1\right)
}\left( z\right) }{f_{q-1,1}\left( z\right) }.  \label{5.13}
\end{equation}%
Since $\sigma _{2}\left( f_{0,j},0\right) $ and by elementary order
considerations we get $\sigma _{2}\left( f_{q-1,1},0\right) \leq \beta ,$
and by Theorem \ref{th1}, we obtain%
\begin{equation}
\left\vert \frac{f_{q-1,1}^{\left( i-j-1\right) }\left( z\right) }{%
f_{q-1,1}\left( z\right) }\right\vert \leq \exp \left\{ \frac{1}{r^{\beta
+\varepsilon }}\right\} ,\ r=\left\vert z\right\vert \notin E.  \label{5.13b}
\end{equation}%
From (\ref{5.11})-(\ref{5.13b}), we get%
\begin{equation}
\left\vert A_{q,j}\left( z\right) \right\vert \leq \exp \left\{ \frac{1}{%
r^{\max \left\{ \tau _{q+j},\beta \right\} +\varepsilon }}\right\} ,\
r\notin E.  \label{5.14}
\end{equation}%
This proves the induction step, and therefore completes the proof of Lemma %
\ref{lem11}.
\end{proof}

\begin{lemma}
\label{lem12}Under the assumptions of Lemma \ref{lem11}, we have%
\begin{equation}
A_{q,0}=A_{0,q}+G_{q}\left( z\right) ,  \label{5.19}
\end{equation}%
where $G_{q}\left( z\right) =\sum\limits_{j=2}^{q+1}H_{j}$ with 
\begin{equation}
H_{j}=\sum\limits_{i=j}^{k-q+j-1}\left( 
\begin{array}{c}
i \\ 
j-1%
\end{array}%
\right) A_{q-j+1,i}\left( z\right) \frac{f_{q-j+1,1}^{\left( i-j+1\right)
}\left( z\right) }{f_{q-j+1,1}\left( z\right) }.  \label{5.19b}
\end{equation}%
Moreover, $G_{q}\left( z\right) $ satisfies%
\begin{equation}
\left\vert G_{q}\left( z\right) \right\vert \leq \exp \left\{ \frac{1}{%
r^{\max \left\{ \tau _{q+1},\beta \right\} +\varepsilon }}\right\} ,\
r=\left\vert z\right\vert \notin E.  \label{5.20}
\end{equation}
\end{lemma}

\begin{proof}
(\ref{5.19}) and (\ref{5.19b}) are the same in \cite[Lemma 6.5]{4}. So, we
need only to prove (\ref{5.20}). We have%
\begin{equation*}
\left\vert G_{q}\left( z\right) \right\vert \leq
\sum\limits_{j=2}^{q+1}\sum\limits_{i=j}^{k-q+j-1}\left( 
\begin{array}{c}
i \\ 
j-1%
\end{array}%
\right) \left\vert A_{q-j+1,i}\left( z\right) \right\vert \left\vert \frac{%
f_{q-j+1,1}^{\left( i-j+1\right) }\left( z\right) }{f_{q-j+1,1}\left(
z\right) }\right\vert .
\end{equation*}

By applying (\ref{5.8}) for the coefficients $\left\vert A_{q-j+1,i}\left(
z\right) \right\vert $ and Theorem \ref{th1} for the logarithmic derivatives 
$\left\vert \frac{f_{q-j+1,1}^{\left( i-j+1\right) }\left( z\right) }{%
f_{q-j+1,1}\left( z\right) }\right\vert $ by taking account that $\sigma
_{2}\left( f_{q-j+1,1},0\right) \leq \beta ,$ we obtain (\ref{5.20}).
\end{proof}

\section{Proof of theorems}

\begin{proof}[Proof of Theorem \protect\ref{th1}]
Suppose that $f$ is meromorphic function in $D\left( 0,R\right) $ with a
singular point at the origin. By Valiron's decomposition lemma we have 
\begin{equation}
f\left( z\right) =z^{m}\phi \left( z\right) \mu \left( z\right)  \label{s1}
\end{equation}%
where\newline
a) The poles and zeros of $f$ in $D\left( 0,R^{\prime }\right) $ are
precisely the poles and zeros of $\phi \left( z\right) .$ The poles and
zeros of $f$ in $D\left( R^{\prime },R\right) $ are precisely the poles and
zeros of $\mu \left( z\right) .$\newline
b) $\phi \left( z\right) $ is meromorphic in $D\left( 0,\infty \right] $\
and analytic and nonzero in $D\left[ R^{\prime },\infty \right] .$\newline
c) $\mu \left( z\right) $ is meromorphic in $D\left( R\right) $ and analytic
and nonzero in $D\left( R^{\prime }\right) .$\newline
Set $\tilde{\phi}\left( z\right) =z^{m}\phi \left( z\right) .$ we have%
\begin{equation*}
\frac{f^{\prime }\left( z\right) }{f\left( z\right) }=\frac{\tilde{\phi}%
^{\prime }\left( z\right) }{\tilde{\phi}\left( z\right) }+\frac{\mu ^{\prime
}\left( z\right) }{\mu \left( z\right) };
\end{equation*}%
and thus%
\begin{equation}
\left\vert \frac{f^{\prime }\left( z\right) }{f\left( z\right) }\right\vert
\leq \left\vert \frac{\tilde{\phi}^{\prime }\left( z\right) }{\tilde{\phi}%
\left( z\right) }\right\vert +\left\vert \frac{\mu ^{\prime }\left( z\right) 
}{\mu \left( z\right) }\right\vert .  \label{l5a}
\end{equation}%
Since $\mu \left( z\right) $ is analytic and non zero in $D\left( R^{\prime
}\right) ,$ we have%
\begin{equation}
\left\vert \frac{\mu ^{\left( j\right) }\left( z\right) }{\mu \left(
z\right) }\right\vert \leq M,\ \left( j\in 
\mathbb{N}
\right) .  \label{l5b}
\end{equation}%
Set $g\left( w\right) =\tilde{\phi}\left( \frac{1}{w}\right) $. Since $\phi
\left( z\right) $ satisfy b), $g\left( w\right) $ is meromorphic in $%
\mathbb{C}
$. We have $\tilde{\phi}\left( z\right) =g\left( w\right) $ such that $w=%
\frac{1}{z};$ then $\tilde{\phi}^{\prime }\left( z\right) =\frac{-1}{z^{2}}%
g^{\prime }\left( w\right) $ and then 
\begin{equation}
\frac{\tilde{\phi}^{\prime }\left( z\right) }{\tilde{\phi}\left( z\right) }=%
\frac{-1}{z^{2}}\frac{g^{\prime }\left( w\right) }{g\left( w\right) }.
\label{l5}
\end{equation}%
By Lemma \ref{lem1}, there exists a set $E_{1}\subset \left( 1,\infty
\right) $ that has a finite logarithmic measure such that for all $%
\left\vert w\right\vert =\frac{1}{\left\vert z\right\vert }=\frac{1}{r}$
satisfying $\frac{1}{r}\notin \left[ 0,1\right) \cup E_{1}$, we have%
\begin{equation*}
\left\vert \frac{g^{\prime }\left( w\right) }{g\left( w\right) }\right\vert
\leq C\left[ T\left( \frac{\alpha }{r},g\right) r\log ^{\alpha }\left( \frac{%
1}{r}\right) \log T\left( \frac{\alpha }{r},g\right) \right] ,\ \frac{1}{r}%
\notin E_{1},
\end{equation*}%
and by Lemma \ref{lem2} and (\ref{l5}),\ we get%
\begin{equation}
\left\vert \frac{\tilde{\phi}^{\prime }\left( z\right) }{\tilde{\phi}\left(
z\right) }\right\vert \leq C\left[ \frac{1}{r}T_{0}\left( \frac{r}{\alpha },%
\tilde{\phi}\right) \log ^{\alpha }\left( \frac{1}{r}\right) \log
T_{0}\left( \frac{r}{\alpha },\tilde{\phi}\right) \right] ,\ r\notin
E_{1}^{\ast };  \label{l5c}
\end{equation}%
where $\frac{1}{r}=R\notin E_{1}\Leftrightarrow r\notin E_{1}^{\ast }$ and $%
\int\limits_{0}^{r_{0}}\frac{\chi _{E_{1}^{\ast }}}{t}dt=\int%
\limits_{1/r_{0}}^{\infty }\frac{\chi _{E_{1}}}{T}dT<\infty ,$ (the constant 
$C>0$ is not the same at each occurrence). Combining (\ref{l5a})-(\ref{l5b})
with (\ref{l5c}) and by taking account Remark \ref{rem1}, we get%
\begin{equation*}
\left\vert \frac{f^{\prime }\left( z\right) }{f\left( z\right) }\right\vert
\leq C\left[ \frac{1}{r}T_{0}\left( \frac{r}{\alpha },f\right) \log ^{\alpha
}\left( \frac{1}{r}\right) \log T_{0}\left( \frac{r}{\alpha },f\right) %
\right] ,\ r\notin E_{1}^{\ast }.
\end{equation*}

We have $\tilde{\phi}^{\prime \prime }\left( z\right) =\frac{1}{z^{4}}%
g^{\prime \prime }\left( w\right) +\frac{2}{z^{3}}g^{\prime }\left( w\right) 
$; and so%
\begin{equation*}
\frac{\tilde{\phi}^{\prime \prime }\left( z\right) }{\tilde{\phi}\left(
z\right) }=\frac{1}{z^{4}}\frac{g^{\prime \prime }\left( w\right) }{g\left(
w\right) }+\frac{2}{z^{3}}\frac{g^{\prime }\left( w\right) }{g\left(
w\right) }.
\end{equation*}%
and by Lemma \ref{lem1} and Lemma \ref{lem2},\ we obtain%
\begin{equation}
\left\vert \frac{\tilde{\phi}^{\prime \prime }\left( z\right) }{\tilde{\phi}%
\left( z\right) }\right\vert \leq C\left[ \frac{1}{r}T_{0}\left( \frac{r}{%
\alpha },\tilde{\phi}\right) \log ^{\alpha }\left( \frac{1}{r}\right) \log
T_{0}\left( \frac{r}{\alpha },\tilde{\phi}\right) \right] ^{2}\ r\notin
E_{1}^{\ast }.  \label{ll1}
\end{equation}%
We have%
\begin{equation}
\frac{f^{\prime \prime }\left( z\right) }{f\left( z\right) }=\frac{\tilde{%
\phi}^{\prime \prime }\left( z\right) }{\tilde{\phi}\left( z\right) }+\frac{%
\mu ^{\prime \prime }\left( z\right) }{\mu \left( z\right) }+2\frac{\tilde{%
\phi}^{\prime }\left( z\right) }{\tilde{\phi}\left( z\right) }\frac{\mu
^{\prime }\left( z\right) }{\mu \left( z\right) }.  \label{l5d}
\end{equation}%
Combining (\ref{ll1})-(\ref{l5d}) with (\ref{l5b}) and by Remark \ref{rem1},
we get%
\begin{equation*}
\left\vert \frac{f^{\prime \prime }\left( z\right) }{f\left( z\right) }%
\right\vert \leq C\left[ \frac{1}{r}T_{0}\left( \frac{r}{\alpha },f\right)
\log ^{\alpha }\left( \frac{1}{r}\right) \log T_{0}\left( \frac{r}{\alpha }%
,f\right) \right] ^{2},\ r\notin E_{1}^{\ast }.
\end{equation*}%
In general, we can find that 
\begin{equation*}
\tilde{\phi}^{\left( k\right) }\left( z\right) =\frac{1}{z^{2k}}g^{\left(
k\right) }\left( w\right) +\frac{a_{k-1}}{z^{2k-1}}g^{\left( k-1\right)
}\left( w\right) +...+\frac{a_{1}}{z^{k+1}}g^{\prime }\left( w\right) ;
\end{equation*}%
where $a_{1},...,a_{k-1}$ are integers; thus%
\begin{equation}
\frac{\tilde{\phi}^{\left( k\right) }\left( z\right) }{\tilde{\phi}\left(
z\right) }=\frac{1}{z^{2k}}\frac{g^{\left( k\right) }\left( w\right) }{%
g\left( w\right) }+\frac{a_{k-1}}{z^{2k-1}}\frac{g^{\left( k-1\right)
}\left( w\right) }{g\left( w\right) }+...+\frac{a_{1}}{z^{k+1}}\frac{%
g^{\prime }\left( w\right) }{g\left( w\right) }.  \label{l6}
\end{equation}%
Also by making use Lemma \ref{lem1} and Lemma \ref{lem2} with (\ref{l6}), we
get,\ for $r=\left\vert z\right\vert <r_{0}$, 
\begin{equation}
\left\vert \frac{\tilde{\phi}^{\left( k\right) }\left( z\right) }{\tilde{\phi%
}\left( z\right) }\right\vert \leq C\left[ \frac{1}{r}T_{0}\left( \frac{r}{%
\alpha },\tilde{\phi}\right) \log ^{\alpha }\left( \frac{1}{r}\right) \log
T_{0}\left( \frac{r}{\alpha },\tilde{\phi}\right) \right] ^{k}\ r\notin
E_{1}^{\ast }.  \label{l6a}
\end{equation}%
We can generalize the equality of $\frac{f^{\left( k\right) }\left( z\right) 
}{f\left( z\right) }$ as follows%
\begin{equation}
\frac{f^{\left( k\right) }\left( z\right) }{f\left( z\right) }%
=\sum\limits_{j=0}^{j=k}\left( 
\begin{array}{c}
j \\ 
k%
\end{array}%
\right)\frac{\tilde{\phi}^{\left( k-j\right)
}\left( z\right) }{\tilde{\phi}\left( z\right) }\frac{\mu ^{\left( j\right)
}\left( z\right) }{\mu \left( z\right) },  \label{ll2}
\end{equation}%
where $\left( 
\begin{array}{c}
j \\ 
k%
\end{array}%
\right)=\frac{k!}{j!\left( k-j\right) !}$ is the binomial
coefficient. Combining (\ref{l6a})-(\ref{ll2}), with (\ref{l5b}) and Remark %
\ref{rem1}, we obtain%
\begin{equation*}
\left\vert \frac{f^{\left( k\right) }\left( z\right) }{f\left( z\right) }%
\right\vert \leq C\left[ \frac{1}{r}T_{0}\left( \frac{r}{\alpha },f\right)
\log ^{\alpha }\left( \frac{1}{r}\right) \log T_{0}\left( \frac{r}{\alpha }%
,f\right) \right] ^{k}\ \left( k\in 
\mathbb{N}
\right) ,
\end{equation*}%
The same reasoning for the case (ii); noting that $\theta \in
E_{2}\Leftrightarrow 2\pi -\theta \in E_{2}^{\ast };\ $so, if $E_{2}\subset %
\left[ 0,2\pi \right) $ has linear measure zero, then $E_{2}^{\ast }\subset %
\left[ 0,2\pi \right) $ has also linear measure zero.
\end{proof}

\begin{proof}[Proof of Theorem \protect\ref{th2b}]
We divide the proof into three parts:\newline
1) If $\sigma _{n}\left( A_{j},0\right) \leq \alpha $ for all $j=0,1,\dots
,k-1,$ then by Lemma \ref{lem7} all solutions $f$ of (\ref{eq3}) satisfy $%
\sigma _{n+1}(f,0)\leq \alpha .$\newline
2) Suppose that $\sigma _{n}\left( A_{j},0\right) =\alpha _{j},$ and let $%
q\in \left\{ 0,1,...,k-1\right\} $ be the largest index such that $\alpha
_{q}=\underset{0\leq j\leq k-1}{\max }\left\{ \alpha _{j}\right\} .$ By Part
1) all solutions $f$ of (\ref{eq3}) satisfy $\sigma _{n+1}(f,0)\leq \alpha
_{q}.$ Assume that there are $q+1$ linearly independent solutions $%
f_{0,1},f_{0,2},...,f_{0,q+1}$ of (\ref{eq3}) satisfy $\sigma
_{n+1}(f_{0,j},0)<\alpha _{q}$ for all $j=1,\dots ,q+1.$ By Lemma \ref{lem11}
with $m=q+1$, there exists a solution $f_{q,1}\not\equiv 0$ of (\ref{5.5})
such that $\sigma _{n+1}(f_{q,1})<\alpha _{q}$ and for any $\varepsilon >0$%
\begin{equation}
\left\vert A_{q,j}\left( z\right) \right\vert \leq \exp _{n}\left\{ \frac{1}{%
r^{\max \left\{ \tau _{q+j},\beta \right\} +\varepsilon }}\right\} ,\
r\notin E.  \label{11}
\end{equation}%
where $\tau _{q+j}=\underset{q+j\leq l\leq k-1}{\max }\left\{ \alpha
_{l}\right\} $ and $j=1,\dots ,k-q-1.$ We have $\max \left\{ \tau
_{q+j},\beta \right\} <\alpha _{q},$ and then%
\begin{equation}
\left\vert A_{q,j}\left( z\right) \right\vert \leq \exp _{n}\left\{ \frac{1}{%
r^{\alpha _{q}-2\varepsilon }}\right\} ,\ r\notin E,  \label{12}
\end{equation}%
for all $j=1,\dots ,k-q-1$ and for $\varepsilon >0$ small enough. Now, by
Lemmas \ref{12}, $\sigma _{n}(A_{q,0},0)=\sigma _{n}(A_{0,q},0)=\alpha _{q}$
and by Lemma \ref{lem6}, there exists a set $F\subset \left( 0,R^{\prime
}\right) $ of infinite logarithmic measure such that for all $r\in F$ we have%
\begin{equation}
\left\vert A_{q,0}\left( z\right) \right\vert \geq \exp _{n}\left\{ \frac{1}{%
r^{\alpha _{q}-\varepsilon }}\right\} ,  \label{13}
\end{equation}%
where $\left\vert A_{q,j}\left( z\right) \right\vert =M_{0}\left(
r,A_{q,j}\right) .$ On the other hand, by (\ref{5.5})%
\begin{equation*}
\left\vert A_{q,0}\left( z\right) \right\vert \leq |\frac{f_{q,1}^{(k-q)}}{%
f_{q,1}}|+|A_{q,k-q-1}(z)||\frac{f_{q,1}^{(k-q-1)}}{f_{q,1}}|+\dots
+|A_{q,1}(z)||\frac{f_{q,1}^{\prime }}{f_{q,1}}|,
\end{equation*}%
and so by (\ref{12}) and Corollary \ref{coro2} with $\sigma
_{n+1}(f_{q,1})<\alpha _{q},$ we get%
\begin{equation}
\left\vert A_{q,0}\left( z\right) \right\vert \leq \exp _{n}\left\{ \frac{1}{%
r^{\alpha _{q}-2\varepsilon }}\right\} ,\ r\notin E.  \label{14}
\end{equation}%
By taking $r\in F\backslash E,$ (\ref{14}) contradicts (\ref{13}). Hence,
there are at most $q$ linearly independent solutions $f$ of (\ref{eq3}) such
that $\sigma _{n+1}(f)<\alpha _{q}.$ Since $\sigma _{n+1}(f)\leq \alpha _{q}$
for all solutions $f$ of (\ref{eq3}), there are at least $k-q$ linearly
independent solutions $f$ of (\ref{eq3}) such that $\sigma _{n+1}\left(
f,0\right) =\alpha _{q}.$\newline
3) Suppose that all solutions $f$ of (\ref{eq3}) satisfy $\sigma
_{n+1}(f,0)\leq \alpha ,$ and assume that there is a coefficient $%
A_{j}\left( z\right) $ of (\ref{eq3}) such that $\sigma _{n}(A_{j})>\alpha .$
If $q\in \left\{ 0,1,...,k-1\right\} $ is the largest index such that $%
\alpha _{q}=\underset{0\leq j\leq k-1}{\max }\left\{ \alpha _{j}\right\} ,$
then by part 2), (\ref{eq3}) has at least $k-q$ linearly independent
solutions $f$ such that $\sigma _{n+1}\left( f,0\right) =\alpha _{q}>\alpha
. $ A contradiction. So, $\sigma _{n}(A_{j})\leq \alpha $ for all $%
j=0,1,\dots ,k-1.$
\end{proof}

\begin{proof}[Proof of Theorem \protect\ref{th2}]
From \eqref{eq3}, we can write 
\begin{equation}
|A_{0}(z)|\leq |\frac{f^{(k)}}{f}|+|A_{k-1}(z)||\frac{f^{(k-1)}}{f}|+\dots
+|A_{1}(z)||\frac{f^{\prime }}{f}|.  \label{p7}
\end{equation}%
\textbf{Case (i):} $\sigma \left( A_{j},0\right) <\sigma \left(
A_{0},0\right) <\infty \ \left( j=1,...,k-1\right) .$ Set $\max \{\sigma
(A_{j},0):j\neq 0\}<\beta <\alpha <\sigma (A_{0},0)$. By (\ref{d5}), there
exists $r_{0}>0$ such that for all $r$ satisfying $r_{0}\geq r>0$, we have 
\begin{equation}
|A_{j}(z)|\leq \exp \{\frac{1}{r^{\beta }}\},\quad j=1,2,\dots ,k-1.
\label{p8}
\end{equation}%
By Lemma \ref{lem5}, there exists a set $F\subset \left( 0,R^{\prime
}\right) $ of infinite logarithmic measure such that for all $r\in F$, we
have 
\begin{equation}
|A_{0}(z)|>\exp \{\frac{1}{r^{\alpha }}\},  \label{p9}
\end{equation}%
where $|A_{0}(z)|=M_{0}\left( r,A_{0}\right) .$ From Theorem \ref{th1},
there exists a set $E_{1}^{\ast }\subset \left( 0,R^{\prime }\right) $ that
has finite logarithmic measure and a constant $C>0$ such that for all $%
r=\left\vert z\right\vert $ satisfying $r\in \left( 0,R^{\prime }\right)
\backslash E_{1}^{\ast }$, we have%
\begin{equation}
\left\vert \frac{f^{\left( j\right) }\left( z\right) }{f\left( z\right) }%
\right\vert \leq \frac{C}{r^{2k}}\left[ T_{0}\left( \frac{r}{\alpha }%
,f\right) \right] ^{2k}\ \left( j=1,...,k-1\right) .  \label{p9b}
\end{equation}%
Using \ref{p8}--\ref{p9b} in \ref{p7}, for $r\in F\backslash E_{1}^{\ast },$
we obtain 
\begin{equation}
\exp \{\frac{1}{r^{\alpha }}\}\leq \frac{C}{r^{2k}}\left[ T_{0}\left( \frac{r%
}{\alpha },f\right) \right] ^{2k}\exp \{\frac{1}{r^{\beta +\varepsilon }}\}.
\label{p10}
\end{equation}%
From \eqref{p10}, we obtain that $\sigma _{2}(f,0)\geq \alpha $.

On the other hand, applying Lemma \ref{lem7} with \eqref{eq3}, we obtain
that $\sigma _{2}(f,0)\leq \sigma (A_{0},0)$. Since $\alpha \leq \sigma
_{2}(f,0)\leq \sigma (A_{0},0)$ holds for all $\alpha <\sigma (A_{0},0)$,
then $\sigma _{2}(f,0)=\sigma (A_{0},0)$.

\textbf{Case (ii):} $0<\sigma \left( A_{j},0\right) \leq \sigma \left(
A_{0},0\right) <\infty $ and $\max \left\{ \tau _{M}\left( A_{j},0\right)
:\sigma \left( A_{j},0\right) =\sigma \left( A_{0},0\right) \right\} <\tau
_{M}\left( A_{0},0\right) \ \left( j=1,...,k-1\right) .$ Set $\max \left\{
\tau _{M}\left( A_{j},0\right) :\sigma \left( A_{j},0\right) =\sigma \left(
A_{0},0\right) \right\} <\mu <\nu <\tau _{M}\left( A_{0},0\right) $ and $%
\sigma (A_{0},0)=\sigma .$ By (\ref{d5b}), there exists $r_{0}>0$ such that
for all $r$ satisfying $r_{0}\geq r>0$, we have%
\begin{equation}
|A_{j}(z)|\leq \exp \{\frac{\mu }{r^{\sigma }}\},\quad j=1,2,\dots ,k-1.
\label{p11}
\end{equation}%
By Lemma \ref{lem5}, there exists a set $F\subset \left( 0,R^{\prime
}\right) $ of infinite logarithmic measure such that for all $r\in F$ and $%
|A_{0}(z)|=M_{0}\left( r,A_{0}\right) ,$ we have%
\begin{equation}
|A_{0}(z)|>\exp \{\frac{\nu }{r^{\sigma }}\}.  \label{p12}
\end{equation}%
Combining (\ref{p11})-(\ref{p12}) with (\ref{p9b}) and (\ref{p7}), we get
for $r\in F\backslash E_{1}^{\ast },$%
\begin{equation}
\exp \{\frac{\nu }{r^{\sigma }}\}\leq \frac{C}{r^{2k}}\left[ T_{0}\left( 
\frac{r}{\alpha },f\right) \right] ^{2k}\exp \{\frac{\mu }{r^{\sigma }}\}.
\label{p13}
\end{equation}%
From (\ref{p13}), we get $\sigma _{2}(f,0)\geq \sigma ,$ and combining this
with Lemma \ref{lem7}, we obtain that $\sigma _{2}(f,0)=\sigma (A_{0},0)$.
\end{proof}

\begin{proof}[Proof of Theorem \protect\ref{th3}]
We begin with the case $a=cb$ $\left( 0<c<1\right) .$ It is easy to see that 
$\tau _{M}\left( A\left( z\right) \exp \left\{ \frac{a}{z^{n}}\right\}
,0\right) =\left\vert a\right\vert $ and $\tau _{M}\left( B\left( z\right)
\exp \left\{ \frac{b}{z^{n}}\right\} ,0\right) =\left\vert b\right\vert .$
By Theorem \ref{th2} case (ii), we get $\sigma _{2}(f,0)=n.$ Now, suppose
that $\arg a\neq \arg b.$ Then, there exist $(\varphi _{1},\varphi
_{2})\subset \lbrack 0,2\pi )$ such that for $\arg (z)=\varphi \in (\varphi
_{1},\varphi _{2}),$ we have $\delta _{b}(\varphi )>0$ and $\delta
_{a}(\varphi )<0$. From \eqref{eq4}, we can write%
\begin{equation}
|B(z)\exp \left\{ \frac{b}{z^{n}}\right\} |\leq |\frac{f^{\prime \prime }}{f}%
|+|A(z)\exp \left\{ \frac{a}{z^{n}}\right\} ||\frac{f^{\prime }}{f}|.
\label{p1}
\end{equation}%
Since $\max \{\sigma (A,0),\sigma (B,0)\}<n$, then by Lemma \ref{lem8}, %
\eqref{l3} and \eqref{p1}, we obtain%
\begin{equation}
\exp \left\{ \left( 1-\varepsilon \right) \delta _{b}\left( \varphi \right) 
\frac{1}{r^{n}}\right\} \leq \frac{C}{r^{4}}\left[ T_{0}\left( \frac{r}{%
\alpha },f\right) \right] ^{4}\exp \left\{ \left( 1-\varepsilon \right)
\delta _{a}\left( \varphi \right) \frac{1}{r^{n}}\right\} .  \label{p14}
\end{equation}%
From (\ref{p14}) we get $\sigma _{2}(f,0)\geq n$ and combining this with
Lemma \ref{lem7}, we obtain that $\sigma _{2}(f,0)=n$.
\end{proof}


\begin{thebibliography}{99}
\bibitem{bieb} L. Bieberbach; \emph{Theorie der gew\"{o}hnlichen
Differentialgleichungen,} Springer-Verlag, Berlin/Heidelberg/New York, 1965.

\bibitem{chen1} Z. X. Chen; \emph{The growth of solutions of $f^{\prime
\prime }+e^{-z}f^{\prime }+Q(z)f=0$, where the order $(Q)=1$}, Sci, China
Ser. A, \textbf{45} (2002), 290-300.

\bibitem{chen3} Z. X. Chen and C. C. Yang; \emph{Some further results on
zeros and growths of entire solutions of second order linear differential
equations, }Kodai Math. J., 22 (1999), 273-285.

\bibitem{fet} H. Fettouch and S. Hamouda; \emph{Growth of local solutions to
linear differential equations around an isolated essential singularity,}
Electron. J. Differential Equations, Vol 2016 (2016), No. 226, pp. 1-10.

\bibitem{gund} G. G. Gundersen; \emph{Estimates for the logarithmic
derivative of a meromorphic function, plus similar estimates}, J. Lond.
Math. Soc. (2), 37 (1988), 88-104.

\bibitem{4} G. G. Gundersen, M. Steinbart and S. Wang; \textit{The possible
orders of solutions of linear differential equations with polynomial
coefficients,} Trans. Amer. Math. Soc. 350 (1998), 1225-1247.

\bibitem{ham12} S. Hamouda; \emph{Properties of solutions to linear
differential equations with analytic coefficients in the unit disc},
Electron. J. Differential Equations, Vol 2012 (2012), No. 177, pp. 1-9.

\bibitem{ham} S. Hamouda; \emph{Iterated order of solutions of linear
differential equations in the unit disc}, Comput. Methods Funct. Theory, 13
(2013) No. 4, 545-555.

\bibitem{hamj} S. Hamouda; \emph{The possible orders of growth of solutions
to certain linear differential equations near a singular point,} J. Math.
Anal. Appl. 458 (2018) 992--1008.

\bibitem{haym} W. K. Hayman; \emph{Meromorphic functions}, Clarendon Press,
Oxford, 1964.

\bibitem{heit} J. Heittokangas, R. Korhonen and J. Rataya; \emph{Fast
Growing Solutions of Linear Differential Equations in the Unit Disc,}
Result.Math. 49 (2006), 265--278.

\bibitem{khri} A.Ya. Khrystiyanyn, A. A. Kondratyuk; \emph{On the Nevanlinna
theory for meromorphic functions on annuli}, Matematychni Studii 23 (1)
(2005) 19--30.

\bibitem{kond} A. A. Kondratyuk, I. Laine; \emph{Meromorphic functions in
multiply connected domains,} in: Fourier Series Methods in Complex Analysis,
in: Univ. Joensuu Dept. Math. Rep. Ser., vol. 10, Univ. Joensuu, Joensuu,
2006, pp. 9-111.

\bibitem{korh} R. Korhonen; \emph{Nevanlinna theory in an annulus, }in:
Value Distribution Theory and Related Topics, in: Adv. Complex Anal. Appl.,
vol. 3, Kluwer Acad. Publ., Boston, MA, 2004, pp. 167-179.

\bibitem{lain} I. Laine; \emph{Nevanlinna theory and complex differential
equations,} W. de Gruyter, Berlin, 1993.

\bibitem{mark} E. L. Mark, Y. Zhuan; \emph{Logarithmic derivatives in
annulus, }J. Math. Anal. Appl. 356 (2009) 441-452.

\bibitem{tu} J. Tu, C-F. Yi; \emph{On the growth of solutions of a class of
higher order linear differential equations with coefficients having the same
order, }J. Math. Anal. Appl. 340 (2008) 487--497.

\bibitem{val} G. Valiron; \emph{Lectures on the General Theory of Integral
Functions,} Chelsea Publishing Company, New York, 1949.

\bibitem{yang} L. Yang; \emph{Value distribution theory,} Springer-Verlag
Science Press, Berlin-Beijing. 1993.
\end{thebibliography}
\end{document}